\documentclass[12pt]{amsart}
\usepackage{amsmath,amsfonts,euscript,amscd,amsthm,amssymb,upref,graphics,color}

\theoremstyle{plain}
\newtheorem{theorem}{Theorem}
\newtheorem{proposition}[theorem]{Proposition}
\newtheorem{lemma}[theorem]{Lemma}

\theoremstyle{definition}
\newtheorem{definition}[subsection]{Definition}

\newtheorem{remark}[subsection]{Remark}

\newtheorem{nothing*}[subsection]{}

\newcommand{\rien}[1]{}

\newcommand{\cO}{{\ensuremath{\mathcal{O}}}}

\newcommand{\cR}{{\ensuremath{\mathcal{R}}}}

\renewcommand{\epsilon}{\varepsilon}
\renewcommand{\phi}{\varphi}

\begin{document}
\renewcommand{\baselinestretch}{1.07}
%%%%%%  TOPMATTER:   %%%%%%%%%%%%%%%%%%%%%%%%%

\title[Punctured Limit Sets]
{Fatou Components with Punctured Limit Sets}

\author{Luka Boc-Thaler}
\author{John Erik Forn\ae ss}
\author{Han Peters}

\date{April 15, 2013}

\vfuzz=2pt
%%%%%%%%%%%%%%%%%%%%%%%%%%%%%%%%%%%%%%%%%%%%%%%%%%%%%%%%%%%%%%%%%%%
%%%%%%%%%%%%%%%%%%%%%%%%%%%%%%%%%%%%%%%%%%%%%%%%%%%%%%%%%%%%%%%%%%%
%%%%%%%%%%%%%%%%%%%%%%%%%%%%%%%%%%%%%%%%%%%%%%%%%%%%%%%%%%%%%%%%%%%

\vskip 1cm

\begin{abstract}
We study invariant Fatou components for holomorphic endomorphisms in $\mathbb{P}^2$. In the recurrent case these components were classified by Sibony and the second author in 1995. In 2008 Ueda completed this classification by proving that it is not possible for the limit set to be a punctured disk. Recently Lyubich and the third author classified non-recurrent invariant Fatou components, under the additional hypothesis that the limit set is unique. Again all possibilities in this classification were known to occur, except for the punctured disk. Here we show that the punctured disk can indeed occur as the limit set of a non-recurrent Fatou component. We provide many additional examples of holomorphic and polynomial endomorphisms of $\mathbb{C}^2$ with non-recurrent Fatou components on which the orbits converge to the regular part of arbitrary analytic sets.
\end{abstract}

\maketitle \vfuzz=2pt

\section{Introduction}

Let $F$ be a holomorphic endomorphism of $\mathbb{P}^2$. Recall that $z \in \mathbb{P}^2$ lies in the \emph{Fatou set} if there exists a neighborhood $U(z)$ on which the family of iterates $\{F^n\}$ is normal. A connected component of the Fatou set is called a \emph{Fatou component}.

Fatou components for rational functions acting on the Riemann sphere have been precisely described. Sullivan \cite{Sullivan1985} proved in 1985 that every Fatou component is (pre-) periodic. Invariant Fatou components were classified by Fatou, who showed that an invariant Fatou component is either the basin of an attracting fixed point, the basin of a parabolic fixed point, or a rotation domain conformally equivalent to either the unit disk or an annulus. The existence of rotation domains equivalent to a disk was later shown by Siegel, and the existence of rotation domains equivalent to an annulus was shown by Herman in 1979.

Fatou components in higher dimensional projective space were studied in  \cite{FS1994, FS1995b, HP1994, JL2003, LP2012, Ueda1994, U2008, Weickert2003}. The following definition is due to Bedford and Smillie \cite{BS1991a}.

\begin{definition}
An invariant Fatou component $\Omega$ is called \emph{recurrent} if there exists an orbit in $\Omega$ with an accumulation point in $\Omega$.
\end{definition}

It follows that if $\Omega$ is non-recurrent then all orbits in $\Omega$ converge to the boundary $\partial \Omega$. Recurrent Fatou components were classified in 1995 by Sibony and the second author \cite{FS1995}:

\begin{theorem}[Forn{\ae}ss-Sibony]\label{thm:recurrent}
Let $F$ be a holomorphic endomorphism with a recurrent invariant Fatou component $\Omega$. Then one of the following holds.
\begin{enumerate}
\item $\Omega$ is the basin of an attracting fixed point $p \in \Omega$.
\item \label{case:two} All orbits in $\Omega$ converge to a closed invariant $1$-dimensional submanifold $\Sigma \subset \Omega$. The map $F$ acts on $\Sigma$ as an irrational rotation, and $\Sigma$ is biholomorphically equivalent to either the disk, the punctured disk, or an annulus.
\item $\Omega$ is a \emph{Siegel domain}: There exists a sequence $(n_j)$ so that $F^{n_j}$ converges uniformly on compact subsets of $\Omega$ to the identity map.
\end{enumerate}
\end{theorem}

The punctured disk in Case \eqref{case:two} was ruled out in 2008 by Ueda \cite{U2008}.

\begin{theorem}[Ueda]\label{thm:Ueda}
The invariant submanifold $\Sigma$ in Case \eqref{case:two} of Theorem \ref{thm:recurrent} cannot be equivalent to a punctured disk.
\end{theorem}
All other cases are known to occur, which makes the classification of recurrent Fatou components for holomorphic endomorphisms of $\mathbb{P}^2$ complete.

The situation is more complicated in the non-recurrent case. If $\Omega$ is a non-recurrent invariant Fatou component then by normality there exists an increasing sequence $(n_j)$ such that the maps $f^{n_j}$ converge to a limit map $h: \Omega \rightarrow \partial \Omega$. The main difficulty in dealing with non-recurrent Fatou components arises because it is not known whether the limit set $h(\Omega)$ is always of the sequence $(n_j)$. If $h(\Omega)$ is independent of $(n_j)$ then we say that $\Omega$ has a \emph{unique limit set}. The following was proved in 2012 by Lyubich and the third author \cite{LP2012}.

\begin{theorem}[Lyubich-Peters]\label{thm:MishaHan}
Let $F$ be a holomorphic endomorphism of $\mathbb{P}^2$ of degree at least $2$, and let $\Omega$ be a non-recurrent invariant Fatou component with a unique limit set. Then $h(\Omega)$ is either a fixed point, or $h(\Omega)$ is equivalent to the unit disk, an annulus, or a punctured disc, and $F$ acts on $h(\Omega)$ as an irrational rotation.
\end{theorem}

Examples where $h(\Omega)$ is a fixed point, a rotating disk or a rotating annulus were known to exist, but whether the punctured disk could also exist remained open. In light of the aforementioned result by Ueda one might expect the punctured disk not to exist in the non-recurrent case either. In Theorem \ref{main} we give an explicit construction of a non-recurent Fatou component with a unique limit set equivalent to a punctured disk.

In the last section we will give a general construction for many more punctured limit sets.

\begin{theorem}\label{thm:regular}
Let $V \in \mathbb{C}^2$ be a pure one-dimensional analytic set. Then there exist a holomorphic endomorphism $F$ of $\mathbb{C}^2$ such that for every irreducible component $V_1$ of $V$ the map $F$ has a non-recurrent Fatou component $\Omega$ on which all orbits converge to $V_1 \setminus \mathrm{Sing}(V)$.
\end{theorem}

The idea is the following. We will construct a map $F = \mathrm{Id} + G$, where $G$ vanishes on the analytic set $V$. By our construction the map $F$ will have a \emph{parabolic curve} attached to each point $(z,w) \in  \mathrm{Reg}(V)$, and these curves vary continuously with $(z,w)$. The union of these curves will be contained in an inviariant Fatou component whose orbits converge to $V$. The fact that the singular points of $V$ are not contained in the limit set then follows from the following result from \cite{LP2012}.

\begin{theorem}[Lyubich-Peters]\label{smoothness}
Let $X$ be a $2$-dimensional complex manifold and $F: X \rightarrow X$ a holomorphic endomorphism. Let $\Omega \subset X$ be an invariant Fatou component and suppose that the sequence $(f^{n_j})$ converges uniformly on compact subsets of $\Omega$ to a rank $1$ limit map $h: \Omega \rightarrow \partial \Omega$. Then $h(\Omega)$ is an injectively immersed Riemann surface.
\end{theorem}

The local dynamics of maps tangent to the identity plays an important role in our results. Let us recall the basic notions here. Let $F : (\mathbb{C}^2, p) \rightarrow (\mathbb{C}^2, p)$ be the germ of a holomorphic map. If $DF(p) = \mathrm{Id}$ then we say that $F$ is \emph{tangent to the identity} at the point $p$. In other words, after changing coordinates by a translation $F$ takes the form
\begin{equation*}
F = \mathrm{Id} + F_k + F_{k+1} + \ldots,
\end{equation*}
where each $F_j$ is a homogeneous polynomial of degree $j$, with \emph{order} $k \ge 2$. Following Hakim \cite{Hakim1998} we say that $v \in \mathbb{C}^2$ is a \emph{characteristic direction} for $F$ if there exists a $\lambda \in \mathbb{C}$ so that
\begin{equation*}
F_k(v) = \lambda v,
\end{equation*}
If $\lambda = 0$ then $v$ is said to be \emph{degenerate}, while if $\lambda \neq 0$ then $v$ is \emph{non-degenerate}. An orbit $\{F^n(z)\}$ is said to converge to the origin \emph{tangentially} to $v$ if $F^n(z) \rightarrow 0$ and $[F^n(z)] \rightarrow [v]$ in $\mathbb{P}^1$.

A \emph{parabolic curve} for $F$ tangent to $[v]\in \mathbb{P}^{1}$ is an injective holomorphic map $\varphi: \mathbb{D}\rightarrow\mathbb{C}^2\backslash\{0\}$, satisfying the following properties:
\begin{itemize}
\item $\varphi$ is continuous at $1\in\partial \mathbb{D}$ and $\varphi(1)=0$,
\item $\varphi(\mathbb{D})$ is $F$-invariant and $(F|_{\varphi(\mathbb{D})})^n\rightarrow 0$ uniformly on compact subsets,
\item $[\varphi(\zeta)]\rightarrow[v]$ as $\zeta\rightarrow 1$ in $\mathbb{D}$.
\end{itemize}

\begin{theorem}[Hakim]\label{thm:Hakim} Let $F: (\mathbb{C}^2,0) \rightarrow (\mathbb{C}^2,0)$ be a holomorphic germ tangent to the identity of order $k \ge 2$. Then for any non-degenerate characteristic direction $v$ there exist (at least) $k-1$ parabolic curves for $F$ tangent to $[v]$.
\end{theorem}

The layout of the paper is as follows. In Section (2) we construct a holomorphic endomorphism of $\mathbb{P}^2$ with an invariant Fatou component $\Omega$ where the limit set is a punctured disk in the boundary of $\Omega$. In Section (3) we construct a large class of holomorphic and polynomial maps for which there exist non-recurrent Fatou components with limit sets equal to the regular parts of analytic sets.

\begin{remark}
The endomorphism of $\mathbb{P}^2$ that we construct in Section (2) is a special case of a one-resonant biholomorphism, which were studied by Bracci and Zaitsev in \cite{BZ2013}. We note that our example is not parabolically attracting, and their main result does not hold for our construction. The maps we study in Section (2) are examples of maps tangent to the identity on one-dimensional analytic subset. Such maps were studied in great detail in \cite{BM2003} and \cite{ABT2004}. 
\end{remark}

{\bf Acknowledgement.}
The third author was supported by a SP3-People Marie Curie Actionsgrant in the project Complex Dynamics (FP7-PEOPLE-2009-RG, 248443).

\section{Construction of a punctured disk}

Throughout this section we let $f$ be the polynomial endomorphism of $\mathbb{C}^2$ given by
\begin{equation*}
f(z,w) = (\lambda z + z^3, \lambda^{-1} (w + zw^2) + w^3).
\end{equation*}
Here $\lambda = e^{2\pi i \theta}$, where $\theta \in \mathbb{R} \setminus \mathbb{Q}$ is chosen such that the maps $z \mapsto \lambda z + z^3$ and $w \mapsto \lambda^{-1}w + w^3$ are linearizable in a neighborhood of the origin. Observe that the polynomial map $f$ extends to a holomorphic endomorphism of $\mathbb{P}^2$, given in homogeneous coordinates by
\begin{equation*}
F[Z:W:T] = [\lambda ZT^2 + Z^3: \lambda^{-1} (WT^2 + ZW^2) + W^3: T^3]
\end{equation*}
Our main result is the following.

\begin{theorem}\label{main}
The map $F$ has an invariant Fatou component $\Omega$ on which all orbits converge to an embedded punctured disk $D^\star \subset \partial \Omega$, and $F$ acts on $D^\star$ as an irrational rotation.
\end{theorem}

The set $D$ will actually be a bounded simply connected subset of the $z$-axis, hence it will be sufficient to only consider the map $f$ and work in Euclidean coordinates. We write $(z_n, w_n)=f^n(z_0, w_0)$ and define
\begin{equation*}
\varphi_n(z,w) = (z, \lambda^n w),
\end{equation*}
and
\begin{align*}
G_n & = \varphi_n \circ f \circ \varphi_{n-1}^{-1}\\
 & = (\lambda z+ z^3, w + \lambda^{1-n}  zw^2 + \lambda^{-2n+3} w^3).
\end{align*}
We also write
\begin{equation*}
g_n(w) = w + \alpha_{n-1} w^2 + \lambda^{-2n+3} w^3,
\end{equation*}
where $\alpha_n = \lambda^{-n}z_n$. Since we have chosen $\lambda$ such that the map $z \mapsto \lambda z + z^3$ is linearizable, there exists a local change of coordinates $\eta(z)=z+z^2h(z)$, conjugating our map $z \mapsto \lambda z + z^3$ to the linear $\zeta \mapsto \lambda \zeta$. Let $A>0$ be such that $|h(z)|<A$ for all $z$ sufficiently close to the origin, so that we obtain
\begin{equation*}
|\lambda^n \eta(z)-  \eta (\lambda^{n}z)|=|z^2 h(z)-z^2\lambda^nh(\lambda^n z)|<2A|z|^2.
\end{equation*}
Since we can also bound $|(\eta^{-1})'(z)|< B$  for $z$ sufficiently close to $0$, it follows that
\begin{equation*}
|\alpha_n -z_0|  =| \lambda^{-n} \eta^{-1}( \lambda^n \eta(z_0))-z_0| =| \eta^{-1}( \lambda^n \eta(z_0))-  \eta^{-1}(  \eta (\lambda^{n}z_0))|\leq C | z_0|^2
\end{equation*}
where $C>0$ does not depend on $n$ and $z_0$. Hence for $z_0$ in a sufficiently small disk $D \subset \mathbb{C}_z$ centered at the origin we have that
\begin{equation}\label{eq:smallball}
|\alpha_n - z_0| \le \frac{|z_0|}{2}
\end{equation}
holds for all $n \in \mathbb{N}$.

\begin{lemma}
For all $z_0 \in D \setminus \{0\}$ there exists an open set $\mathcal{C}_{z_0} \subset \{z = z_0\}$ on which
\begin{equation*}
\|f^n(z_0,w) - f^n(z_0, 0)\| \rightarrow 0,
\end{equation*}
uniformly on compact subsets of $\mathcal{C}_{z_0}$.
\end{lemma}
\begin{proof}
Notice that $f^n = \varphi_n^{-1} \circ G_n \circ \ldots G_1$. Since we are interested in the set of $w$-values for which $w_n$ converges to $0$, and the map $\varphi_n$ preserves distances to $(z_n,0)$, it is equivalent to consider the $w$-values for which
\begin{equation*}
\pi_2 (G_n \circ \ldots \circ G_1(z_0,w)) = g_n \circ \ldots \circ g_1(w)
\end{equation*}
converges to $0$. We will write $g(n)$ for the composition $g_n \circ \ldots \circ g_1$. Let us also write
\begin{equation*}
u_n = \frac{1}{g(n)(w)}.
\end{equation*}
We have that
\begin{equation*}
u_{n+1} = \frac{1}{g_{n+1}(\frac{1}{u_n})} = u_n - \alpha_n + O(\frac{1}{|u_n|}).
\end{equation*}
Having chosen $z_0$ sufficiently close to the origin so that Equation \eqref{eq:smallball} holds, it follows that $|u_n| \rightarrow \infty$ if the initial value $u_0$ lies in a halfplane of the form
\begin{equation}\label{halfplane}
\mathbb{H}_{z_0} = \{u \in \mathbb{C} \mid \mathrm{Re}(u \bar{z_0}) < - K(z_0)\}.
\end{equation}
Hence for all $u\in \mathbb{H}_{z_0}$ we have
\begin{equation*}
\|f^{n}(z_0,u^{-1})-f^n (z_0,0)\| \rightarrow 0.
\end{equation*}
The statement of the lemma therefore holds for the set
\begin{equation*}
\mathcal{C}_{z_0} = \left\{(z_0, w) \mid \; \frac{1}{w} \in \mathbb{H}_{z_0} \right\}.
\end{equation*}
\end{proof}

In Equation \eqref{halfplane} the constant $K(z)$ can be chosen to vary continiously with $z$ in the punctured disk $D \setminus \{0\}$, and therefore the sets $\mathcal{C}_{z_0}$ also vary continuously with $z \in D \setminus \{0\}$. Let $U \subset D$ be an $f$-invariant neighborhood of the origin and write
\begin{equation*}
V= \{(z_0, w_0) \mid z_0 \in U \setminus \{0\}, \; w_0 \in C_{z_0} \}.
\end{equation*}
We then define the $f$-invariant open connected set $\Lambda$ by
\begin{equation*}
\Lambda=\bigcup_{n=0}^{\infty}f^n(V).
\end{equation*}

Observe that $\{f^n\}_n$ is normal familly on $\Lambda$, hence $\Lambda$ is contained in some invariant Fatou component $\Omega$.

\begin{lemma}
Every orbit in $\Omega$ converges to the plane $\mathbb{C}_z$.
\end{lemma}
\begin{proof}
Let $(f^{n_j})$ be a sequence that converges uniformly on compact subsets of $\Omega$ to a map $h$. Then $h(\Lambda) \subset \{w = 0\}$. Since $h$ is holomorphic and $\Lambda$ has interior, it follows that $h(\Omega) \subset \{w = 0\}$.
\end{proof}

\begin{lemma}
Let $(f^{n_j})$ be a convergent subsequence with limit $h: \Omega \rightarrow \partial \Omega$. Then $h(\Omega)$ is contained in the Siegel disk in the plane $\mathbb{C}_z$, and is independent of the sequence $(n_j)$.
\end{lemma}
\begin{proof}
Lemma 13 of \cite{LP2012} states that the restriction of the maps $\{f^n\}$ to the set $h(\Omega)$ must also form a normal family. Hence $h(\Omega)$ is contained in the Siegel disk in the $z$-plane centered at the origin, which we call $V$. Suppose that $k: \Omega \rightarrow \bar{\Omega}$ is any other limit map of the sequence $(f^n)$. By the same argument as above we have that $k(\Omega)$ lies in the Siegel disk $V$. Therefore it follows from the skew-product structure of $f$ that $k = \rho \circ h$, where $\rho: V \rightarrow V$ is a limit map of the sequence $(f^n)$ restricted to $V$. By invariance of $\Omega$ it follows that $h(\Omega) = k(\Omega)$.
\end{proof}

\begin{lemma}
The Fatou component $\Omega$ is non-recurrent and the limit set $h(\Omega)$ is a punctured disk.
\end{lemma}
\begin{proof}
It follows from the skew-product structure of $f$, and the fact that the restriction of $f$ to $\{z = 0\}$ is linearizable in a neighborhood of the origin, that no orbits in $\Omega$ converge to $(0,0)$. But $h(\Omega)$ does contain any point $(z,0)$ with $z \neq 0$ sufficiently small. Hence $h(\Omega)$ is a $1$-dimensional submanifold of $\mathbb{C}^2$ that is not equivalent to either the unit disk or to an annulus. Therefore it follows from Theorems \ref{thm:recurrent} and \ref{thm:Ueda} that $\Omega$ must be a non-recurrent Fatou component, and Theorem \ref{thm:MishaHan} implies that $h(\Omega)$ is an embedded punctured disk.
\end{proof}

With this lemma we have completed the proof of Theorem \ref{main}.

\begin{remark}
One easily sees that the Siegel disk centered at $0$ at in the $z=0$ plane lies in the Julia set. Indeed, suppose that a point $(0,w)$ lies in the Fatou set, and let $U$ be a neighborhood of $(0,w)$ on which the family $\{f^n\}$ is normal. Since $\frac{\partial f^n}{\partial z}(0,w)$ and $\frac{\partial f^n}{\partial w}(0,w)$ are bounded away from $0$, and by our assumption that the iterates $(f^n)$ form a normal family, the union of the forward images of $U$ contains a tubular neighborhood of the $\omega$-limit set of $(0,w)$, which is a Jordan curve in the $z=0$ plane whose interior contains the origin. Hence for $|c|$ sufficiently small the intersection of this tubular neighborhood with the fiber $\{z = c\}$ contains an annulus. But then the family of iterates of $f$ restricted to the area enclosed by this annulus must be bounded, and thus a normal family. But this area includes the origin, which leads to a contradiction.
\end{remark}

\section{Regular limit sets}\label{cusp}

Let $V \subset \mathbb{C}^2$ be any analytic set of pure dimension one. Observe that there exist an open cover $U_n\subset \mathbb{C}^2$ of $V$ and a collection of minimal defining functions $g_n\in\cO(U_n)$ for sets $V\cap U_n$, i.e. $\{g_n=0\}=V\cap U_n$ and $\{g_n=dg_n=0\}=\mathrm{Sing}(V\cap U_n)$. Using the fact that the Cousin II problem is always solvable on any one-dimesional Stein space one can prove the existence of a minimal defining function $g\in\cO(\mathbb{C}^2)$ for the set $V$. Let us define the following map
\begin{equation}\label{map}
\quad\quad\quad\quad\quad\quad\quad F = (z,w) + g^k(z,w)(P(z,w),Q(z,w)),\quad\quad\quad (k\geq 2),
\end{equation}
 where $P$ and $Q$ are holomorphic functions on $\mathbb{C}^2$ and $g$ is aminimal defining function of $V$. Observe that $F$ is tangent to identity on $V$. If we can find $P$ and $Q$ such that for each point in $\mathrm{Reg}(V)$ there exists a non-degenerate characteristic direction, then by Hakim's Theorem $\ref{thm:Hakim}$ there will be a parabolic curve attached to each point $(z,w)$ from the regular part of $V$. These curves vary continuously with the base point. The union of these curves will be contained in an inviariant Fatou component whose orbits converge to $V$. The fact that the singular points of $V$ are not contained in the limit set then follows from Theorem $\ref{smoothness}$.

Let us pick a point $(z_0,w_0)\in\mathrm{Reg}(V)$. After conjugating $F$ with the right translation map we note that the existence of a non-degenerate characteristic direction corresponds to
\begin{equation}\label{condition}
g_z(z_0,w_0)P(z_0,w_0)+g_w(z_0,w_0)Q(z_0,w_0)\neq0,
\end{equation}
where $g_z$ and $g_w$ are the partial derivatives with respect to $z$ and $w$. We see immediately that in order to have non-characteristic directions it is necessary that the gradient of $g$ does not vanish on $\mathrm{Reg}(V)$. Observe that this condition is satisfied if and only if $g$ is a minimal defining function of set $V$. In order to prove Theorem \ref{thm:regular} we first have to prove the existence of holomorphic functions $P$ and $Q$ for which (\ref{condition}) holds.

\begin{definition}Let $V$ be an analytic set and let $\mathrm{Sing}(V)$ be the set of all singular points of $V$. A function $f:V \setminus \mathrm{Sing}(V) \rightarrow \mathbb{C}$ is
said to be \emph{weakly holomorphic} on $V$, if $f$ is holomorphic on $V \setminus \mathrm{Sing}V$ and locally bounded along $\mathrm(V)$. The sheaf of weakly holomorphic functions on $V$ is denoted by $\tilde{\cO}_V$.
\end{definition}

\begin{definition} Let $V$ be an analytic set in some open set $U$ in $\mathbb{C}^n$. A holomorphic function $f$ on $U$  is called a \emph{universal denominator} for $V$ at the point $z\in V$, if  $f_z\cdot\tilde{\cO}_{V,z}\subset \cO_{V,z}$.
\end{definition}

\begin{lemma}\label{lem:PQ}
Let $g$ be the minimal defining function of an analytic set $V\subset\mathbb{C}^2$. There exist holomorphic functions $P$ and $Q$ on $\mathbb{C}^2$ such that $g_zP+g_wQ$ does not vanish on $\mathrm{Reg}(V)$.
\end{lemma}
\begin{proof} The partial derivatives of $g$ with respect to $z$ and $w$ will be denoted by $g_z$ and $g_w$ respectively.
Recall that an analytic set $V=\{g(z,w)=0\}$ is a Stein space. Let us denote by $\cO_V$ the coherent sheaf of holomorphic functions on $V$.

If $V$ is a manifold, then the partial derivatives $g_z|_V$ and $g_w|_V$ do not have common zeros and so they generate $\cO_V$. By Cartan's Division Theorem (Corollary 2.4.4., \cite{For}) there are $p,q\in\cO_V$ such that $p\cdot g_z|_V+q\cdot g_w|_V=1$. Cartan's Extension Theorem (Corollary 2.4.3., \cite{For}) gives that every holomorphic function on $V$ can be extended to a holomorphic function on $\mathbb{C}^2$, which proves the existence of desired holomorphic functions $P$ and $Q$.

In general we can not expect $V$ to be a manifold. Let us denote the (discrete) set of singular points of $V$ by $\{a_n\}_{n\geq1}$. By our assumption on $g$ we have that $\mathrm{Sing}(V)= \{g = dg = 0\}$. Since $V$ is one-dimension variety, it has a normalization given by a non-compact Riemann surface $M$ together with a holomorphic map $\pi:M\rightarrow V$, \cite{Cir}. We can lift $g_z$ and $g_w$ to holomorphic functions on $M$ denoted by $\varphi$ and $\psi$ respectively. Since $M$ is a non-compact Riemann surface, by Weierstrass Theorem (Theorem 26.7., \cite{Fo}), we can find a holomorphic function $\theta$ on $M$ with prescribed zeros, such that $\varphi/\theta$ and $\psi/\theta$ are holomorphic on $M$ and without common zeros. As before, $M$ is a Stein manifold and $\varphi/\theta$, $\psi/\theta$ generate $\cO_M$, so we can find $\tilde{p},\tilde{q}\in\cO_M$ such that $$\frac{\varphi}{\theta}\tilde{p}+\frac{\psi}{\theta}\tilde{q}=1,$$  hence
\begin{equation*}
\varphi\tilde{p}+\psi\tilde{q}=\theta.
\end{equation*}
The functions $\tilde{p}$ and $\tilde{q}$ induce weakly holomorphic functions $p,q\in \tilde{\cO}_V$. For every $a_n\in \mathrm{Sing}(V)$ there exists an integer $m_n>0$ such that $g_z^{m_n}$ and $g_w^{m_n}$ are universal denominators in some neighborhood  $a_n\in U_n\subset\mathbb{C}^2$ (Corollary $3.$ \cite{N1966}, p.$59.$). By expanding $(g_zp+g_wq)^{2m_n+1}$ in to series we obtain
\begin{equation*}
\label{1} (g_zp+g_wq)^{2m_n+1}=g_zP_n+g_wQ_n,
\end{equation*}
with
\begin{align*}
P_n & =g_z^{m_n}\sum_{k=0}^{m_n} \binom{2m_n+1}{k}  p^{2m_n+1-k}g_z^{m_n-k}(qg_w)^k, \; \; \mathrm{and}\\
Q_n & =g_w^{m_n}\sum_{k=0}^{m_n} \binom{2m_n+1}{k}  q^{2m_n+1-k}g_w^{m_n-k}(pg_z)^k.
\end{align*}

Since $g_z^{m_n}$ and $g_w^{m_n}$ are universal denominators we can conclude that $P_n$ and $Q_n$ are holomorphic on $U_n$.
From the construction it follows that $g_zp+g_wq$ is non-zero on $\mathrm{Reg}(V)\cap  U_n$, and the same holds for the function $h_n:=g_zP_n+g_wQ_n$. We can take $\{U_n\}_{n\geq 1}$ to be pairwise disjoint. Let us take one more open set $U_0\subset\mathbb{C}^2$ such that $W_0=U_0\cap V\subset\mathrm{Reg}(V)$ and that $\{W_n=U_n\cap V\}_{n\geq0}$ covers $V$. The function $h_0=g_zp+g_wq$ is non-zero holomorphic on $W_0$.  Note that $\{(h_n,W_n)\}_{n\geq0}$ is a Cousin II distribution  and the Cousin II problem is always solvable on any one-dimensional Stein space (\cite{GR}, p.148). Hence there exists a holomorphic function $H\in \cO_V$ with the property
\begin{equation*}
H|_{W_n}=h_n\varphi_n,
\end{equation*}
where the holomorphic functions $\varphi_n$ are non-zero on $W_n$. If we define $\tilde{P_n}=P_n\varphi_n$ and $\tilde{Q_n}=Q_n\varphi_n$  (where $P_0=p$ and $Q_0=q$) observe that $$H|_{W_n}=\tilde{P_n}g_z+\tilde{Q_n}g_w.$$

Let $\mathcal{J}_V\triangleleft \cO_V $ be the ideal sheaf  generated by  $g_z$ and $g_w$. Since $\mathcal{J}_V$ is a coherent analytic sheaf we are given a short exact sequence
$$0\rightarrow\cR\stackrel{i}{\hookrightarrow}\cO^2_V\stackrel{\tau}{\rightarrow}\mathcal{J}_V\rightarrow 0$$
where $i$ is an inclusion map and $\tau$ maps $(f_1,f_2)$ into $f_1g_z+f_2g_w$. Now we can form a long exact sequence of cohomology grups
$$0\rightarrow \Gamma(V,\cR)\stackrel{i^*}{\rightarrow}\Gamma(V,\cO^2_V)\stackrel{\tau^*}{\rightarrow}\Gamma(V,\mathcal{J}_V)\rightarrow H^1(V,\cR)\rightarrow\cdots$$
and since $V$ is Stein we know that $H^1(V,\cR)=0$, hence $\tau^*$ is surjective. Observe that $H$ induces a section of $\mathcal{J}_V$. Since $\tau^*$ is surjective this section can belifted to a section of $\cO_V^2$ (Section 7.2 \cite{Kra}), hence there exist holomophic functions $P$, $Q$ on $V$ such that $Pg_z+Qg_w=H$ on $V$. By Cartan's Extension Theorem we can extend $P$ and $Q$ to holomorphic functions on $\mathbb{C}^2$.
\end{proof}

From now on let $F$ be of the form given in (\ref{map}), where $P$ and $Q$ are as in Lemma \ref{lem:PQ}.

\begin{lemma}\label{lemma:local}
For any $(z,w) \in \mathrm{Reg}(V)$ we can find a local change of coordinates that maps $(z,w)$ to $(0,0)$ so that in the new coordinates $F$ takes the form
\begin{equation}\label{germ}
\left\{ \begin{aligned}
x_1 & = x - x^k + x^k O(x,u),\\
u_1 & = u + x^k O(x,u).
\end{aligned} \right.
\end{equation}
\end{lemma}
\begin{proof}
Near any regular point of $V$ we can change coordinates so that $V$ is given by $\{(z,w) \mid z = 0\}$. By its definition the map $F$ now takes the form
\begin{equation*}
\left\{ \begin{aligned}
z_1 & = z + z^kp(z,w),\\
w_1 & = w + z^kq(z,w).
\end{aligned} \right.
\end{equation*}
Note that if we write $F = Id + F_k + F_{k+1} + \ldots$, where $F_j$ is homogeneous of degree $j$, then the fact that $dg \neq 0$ on $V$ implies that $F_k$ is non-zero at $(0,0)$. The assumption that $g_zP+g_wQ$ does not vanish on $\mathrm{Reg}(V)$ implies that there is a non-degenerate characteristic direction at each point in $\mathrm{Reg}(V)$, in particular at $(0,0)$. This direction cannot be $(0,1)$, which is degenerate, hence we may write $(1,\lambda)$ for the non-degenerate characteristic direction. We can now find a linear change of coordinates that fixes $V = \{z = 0\}$ and changes the non-degenerate characteristic direction to the tangent vector $(1,0)$. Therefore in these new coordinates $F$ takes on the form
\begin{equation*}
\left\{ \begin{aligned}
x_1 & = x + \beta x^k + x^k O(x,u),\\
u_1 & = u + x^k O(x,u).
\end{aligned} \right.
\end{equation*}
By conjugating with a linear map $(x,u) \mapsto ((- \beta)^{\frac{1}{1-k}}x, u)$ we obtain the required form.
\end{proof}

Let us rewrite Equation \eqref{germ} as
\begin{equation}\label{germ2}
\left\{ \begin{aligned}
x_1 & = x - x^k(1 + u \phi(u)) + O(x^{k+1}),\\
u_1 & = u + x^k O(x,u).
\end{aligned} \right.
\end{equation}
We note that for a fixed value of $u$ the real attracting directions of the map $x \mapsto x_1$ satisfy
\begin{equation*}
\mathrm{Arg}(x) = \mathrm{Arg}(x^k(1 + u\phi(u))) = k \mathrm{Arg}(x) + \mathrm{Arg}(1 + u\phi(u)),
\end{equation*}
hence the $k-1$ attracting real directions satisfy
\begin{equation*}
\mathrm{Arg}(x) = \theta_u(m) = \frac{2m\pi - \mathrm{Arg}(1 + u \phi(u)) }{k-1},
\end{equation*}
for $m = 1 , \ldots , k-1$. For $\epsilon >0 $ sufficiently small we define the regions
\begin{equation*}
R_\epsilon(m) := \left\{(x,u) \mid 0 < |x| < \epsilon, \; |u| < 2 \epsilon, \; |\mathrm{Arg}(x) - \theta_u(m)| < \frac{\pi}{2k-2} \right\}.
\end{equation*}

\begin{lemma}\label{lemma:attracting}
For $\epsilon>0$ sufficiently small the iterates $F^{n}$ converge uniformly on each $R_\epsilon(m)$ to the axis $\{x = 0\}$.
\end{lemma}
\begin{proof}
For a fixed $u$-value the convergence of the $x$-coordinate to $0$ follows from standard estimates. Two estimates guarantee that variation in the $u$-coordinate does not break this convergence. On the one hand we have for sufficiently small $x, u$ that satisfy
\begin{equation}\label{eq:middle}
\frac{1}{3}\frac{\pi}{2k-2} \le \left|\mathrm{Arg}(x) - \theta_u(m)\right| \le  \frac{5}{3}\frac{\pi}{2k-2}
\end{equation}
the estimate
\begin{equation}\label{eq:rotation}
\left|\mathrm{Arg}(x_1) - \theta_{u_1}(m)\right| < \left|\mathrm{Arg}(x) - \theta_u(m)\right|,
\end{equation}
and even
\begin{equation}\label{eq:rotation2}
|x|\left(\left|\mathrm{Arg}(x) - \theta_u(m)\right| - \left|\mathrm{Arg}(x_1) - \theta_{u_1}(m)\right|\right) \gg |u_1 - u|.
\end{equation}
On the other hand we have for sufficiently small $x, u$ that satisfy
\begin{equation}\label{eq:center}
\left|\mathrm{Arg}(x) - \theta_u(m)\right| <  \frac{1}{2}\frac{\pi}{2k-2}
\end{equation}
that
\begin{equation}\label{eq:attract}
|x_1| < |x|,
\end{equation}
and even
\begin{equation}\label{eq:attract2}
|x| - |x_1| \gg |u_1 - u|.
\end{equation}
The statement of the lemma follows.
\end{proof}

Let us write
\begin{equation*}
U_\epsilon = \left\{(x,u) \mid |x| < \epsilon, \; |u| < 2\epsilon \right\},
\end{equation*}
and define the nested sequence
\begin{align*}
W_0(m) & = R_\epsilon(m), \; \; \; \mathrm{and} \\
W_{n+1}(m) & = (W_n(m) \cup f^{-1} W_{n}(m)) \cap U_\epsilon.
\end{align*}
Finally we define
\begin{equation*}
W(m) = \bigcup_{n \in \mathbb{N}} W_n(m).
\end{equation*}
Since $f$ is invertible in a neighborhood of $(0,0)$, for sufficiently small $\epsilon$ the sets $W(m)$ are open, connected and disjoint. Moreover we have the following.

\begin{lemma}\label{lemma:separatrix}
Suppose $x_0$ and $u_0$ satisfy $|x_0| < \epsilon$ and $|u_0| < \epsilon$, and $(x_0, u_0)$ does not lie in one of the sets $W(m)$. Then there is an $N \in \mathbb{N}$ such that the orbit $(z_n) = ((x_n, u_n))$ satisfies
\begin{enumerate}
\item[$\bullet$] For $n \le N$ we have $|x_n| < \epsilon$ and $|u_n| < 2 \epsilon$, and
\item [$\bullet$] $|x_N| \ge \epsilon$.
\end{enumerate}
\end{lemma}
\begin{proof}
Considering Equations \eqref{eq:center}, \eqref{eq:attract} and \eqref{eq:attract2} for $F^{-1}$ gives for sufficiently small $x, u$ that satisfy
\begin{equation*}
\left|\mathrm{Arg}(x) - (\theta_u(m) + \frac{\pi}{k-1})\right| <  \frac{1}{2}\frac{\pi}{2k-2}
\end{equation*}
the estimate
\begin{equation*}
|x_1| > |x|,
\end{equation*}
and even
\begin{equation*}
|x_1| - |x| \gg |u_1 - u|.
\end{equation*}
The combination with Equations \eqref{eq:middle}, \eqref{eq:rotation} and \eqref{eq:rotation2} implies the statements of the lemma.
\end{proof}

We can now complete the proof of the main result of this section.

\begin{proof}[{\bf Proof of Theorem \ref{thm:regular}}]
Let $g$ be a minimal defining function of $V$ and let $P$ and $Q$ be as in Lemma \ref{lem:PQ}. Let $k\geq 2$ and define
\begin{equation*}
F = (z,w) + g^k(z,w)(P(z,w),Q(z,w)).
\end{equation*}
By Lemma \ref{lemma:local} we can find a local change of coordinates near any point in $(z,w)\in\mathrm{Reg}(V)$ to obtain the form \eqref{germ}. In particular by Theorem \ref{thm:Hakim} we obtain $k-1$ parabolic curves on which we have attraction towards (in original coordinates) the point $(z, w)$.

Let us first assume that $\mathrm{Reg}(V)$ is connected. By Lemma \ref{lemma:local} and Lemma \ref{lemma:attracting} there exist, in a neighborhood of the point $(x,y)$, $k-1$ open regions $R_\epsilon(m)$ on which all orbits converge uniformly to $V$. These regions vary smoothly with the point $(x,y)$ and each region $R_\epsilon(m)$ intersects at least one of the $k-1$ parabolic curves at $(x,y)$. Consider the union over all base points $(x,y)$ of the regions $R_\epsilon(m)$, and let $\tilde{\Omega}$ be one of the connected components of this union. Then $F^n$ converges uniformly on $\tilde{\Omega}$ to $\mathrm{Reg}(V)$. Hence $\tilde{\Omega}$ is contained in a Fatou component for $F$, which we denote by $\Omega$.

By the aforementioned non-empty intersection with the parabolic curves, there exist for each point $z \in \mathrm{Reg}(V)$ corresponding points in $\Omega$ whose orbits converge to $z$. Hence on $\Omega$ the iterates $F^n$ converges to a holomorphic map whose image contains $\mathrm{Reg}(V)$. By Lemma \ref{lemma:separatrix} it follows that $\mathrm{Reg}(V)$ cannot be contained in $\Omega$, hence must lie in $\partial \Omega$. Thus $\Omega$ is a non-recurrent Fatou component. By Theorem \ref{smoothness} it follows that the limit set must be exactly equal to $\mathrm{Reg}(V)$.

Now suppose that $\mathrm{Reg}(V)$ is not connected and let $V_1$ be an irreducible component of $V$. By the above argument there is a Fatou component $\Omega$ on which all orbits converge to $V_1 \setminus \mathrm{Sing}(V_1)$. It remains to be shown that there are no points converging to intersection points of $V_1$ with other irreducible components of $V$. Let $V_2 \neq V_1$ be a connected component of $V$ and assume that there is a point $z \in V_1 \cap V_2$. Let us recall an argument that was used in the proof of Theorem \ref{smoothness}. Suppose for the purpose of a contradiction that there exists a point $x \in \Omega$ whose orbit converges to $z$. Denote the limit of the sequence $(F^n)$ on $\Omega$ by $h$. Let $D$ be a holomorphic disk through $x$ so that $h(D) = h(U)$ for a small neighborhood $U$ of $x$. Then $F^n(D)$ intersects $V_2$ for $n$ large enough. But as we noted earlier, $V_2$ lies in the Julia set while $F^n(D)$ lies in the Fatou set, which gives a contradiction. This completes the proof. \end{proof}

Let $\Omega$ be a Fatou component on which the orbits converge to the regular partner of an irreducile component $V_1 \subset V$. Then by Lemmas \ref{lemma:attracting} and \ref{lemma:separatrix} the orbit of any point in $\Omega$ eventually lands in one of the sets $R_\epsilon(m)$. Vice versa, any point whose orbit eventually lands in $R_\epsilon(m)$ is contained in $\Omega$. In particular the parabolic curve at $z$ corresponding to $R_\epsilon(m)$ must be contained in $\Omega$.

\begin{proposition}
The orbit of any point in $\Omega$ is evetually mapped onto a parabolic curves contained in $\Omega$.
\end{proposition}
\begin{proof}
What remains to be shown is that any point whose orbit lands in a set $R_\epsilon(m)$ for $\epsilon$ sufficiently small is contained in parabolic curve. Hence it is sufficient to show that, in the local coordinates given by Equation \eqref{germ}, any point $(x,0)$ with $|x|$ sufficiently small and
\begin{equation}\label{eq:angle}
|\mathrm{Arg}(x) - \theta_u(m)| < \frac{\pi}{2k-2}
\end{equation}
lies on a parabolic curve. By Hakim's proof of Theorem \ref{thm:Hakim} in \cite{Hakim1998} we know that, for $\epsilon>0$ sufficiently small, the parabolic curve at $(0,u)$ is a graph over
\begin{equation*}
\{x \in \mathbb{C} \mid |x| < \epsilon, \; \;|\mathrm{Arg}(x) - \theta_u(m)| < \frac{\pi}{2k-2} \},
\end{equation*}
of the form $v_u(x) = u + x^k h_u(x)$, with $\|h_u\| < 1$. By Hakim's proof these graphs vary continuously with $h_u$, hence for $|x|$ small enough any point $(x,0)$ satisfying Equation \eqref{eq:angle} lies on one of the parabolic curves.
\end{proof}

If $k \ge 3$ then for each point in $\mathrm{Reg}(V)$ there exist $k-1$ parabolic curves. We investigate whether these curves lie in distinct Fatou components.

Note that the parabolic curves vary continuously with the basepoint in $\mathrm{Reg}(V)$. To each of these parabolic curves corresponds to a real tangent vector $\alpha$ for which
\begin{equation}
\mathrm{d}F \alpha = - \lambda \alpha,
\end{equation}
with $\lambda \in \mathbb{R}^+$. The vector $\alpha$ is unique up to multiplication in $\mathbb{R}^+$. Let $z_0, z_1 \in \mathrm{Reg}(V)$ with corresponding parabolic curves $C_0$ and $C_1$ and real attracting tangent vectors $\alpha_0$ and $\alpha_1$. We say that $(z_0, C_0) \sim (z_1, C_1)$ if there exists a continuous map $\phi$ from $[0,1]$ to the set of pairs $(z, \alpha)$, with $z \in \mathrm{Reg}(V)$ and $\alpha$ a real attracting tangent vector to $z$, so that $\phi(0) = (v_0, \alpha_0)$ and $\phi(1) = (v_1, \alpha_1)$. Clearly if $(z_0, C_0) \sim (z_1, C_1)$ then $C_0$ and $C_1$ lie in the same Fatou component. The converse also holds.

\begin{lemma}
The parabolic curves $C_0$ and $C_1$ lie in the same Fatou component if and only if $(z_0, C_0) \sim (z_1, C_1)$.
\end{lemma}
\begin{proof}
We only need to show that if $C_0$ and $C_1$ lie in the same Fatou component, then $(z_0, C_0) \sim (z_1, C_1)$. Assume therefore that $w_1 \in C_1$ and $w_2 \in C_2$, and suppose that $w_1$ and $w_2$ lie in the same Fatou component $\Omega$. Let $\gamma$ be a
continuous real curve in $\Omega$, starting at $w_1$ and ending at $w_2$. Since $\gamma$ is compact we know by Theorem \ref{thm:regular} that the sequence $(F^n)$ converges uniformly on $\gamma$ to a limit set $h(\gamma)$ contained in $\mathrm{Reg}(V)$.

Since $h(\gamma)$ is compact we can find a uniform $\epsilon >0$ for which the statements in Lemmas \ref{lemma:attracting} and \ref{lemma:separatrix} hold. Let $N \in \mathbb{N}$ be such that $\|F^n - h\|_\gamma < \epsilon$ for all $n \ge N$. The curve $F^N(\gamma)$ lies in the invariant Fatou component $\Omega$, and still starts at a point in $C_0$ and ends at a point in $C_1$. It follows from Lemma \ref{lemma:separatrix} that $F^N(\gamma)$ must lie in the union of the open sets $W(m)$. Hence we can follow the real attracting direction of the sets $W(m)$, starting with the direction corresponding to $C_0$ and ending with the direction corresponding to $C_1$.
\end{proof}

We give an simple family of polynomial endomorphisms of $\mathbb{C}^2$ for which we can easily determine whether the parabolic curves lie in distinct Fatou components. We let $p$ and $q$ be relatively prime, $k \ge 2$ be an integer, and define
\begin{equation*}
F(z,w) = (z,w) + (z^p-w^q)^k \cdot(z, -w).
\end{equation*}
As noted in the proof of Theorem \ref{thm:regular}  every point $(t^q, t^p) \in V = \{z^p - w^q = 0\}$ has a non-degenerate characteristic direction $(t^q, -t^p)$, and there are exactly $k-1$ parabolic curves tangent to this characteristic direction. Whether the curves $C_1, \ldots C_{k-1}$ all lie in distinct Fatou components or not depends on the values of $k, p$ and $q$.

\begin{proposition}
If $p\cdot q$ is divisible by $k-1$ then the curves $C_1, \ldots, C_{k-1}$ all lie in distinct Fatou components. If $p\cdot q$ and $k-1$ are relatively prime then $C_1, \ldots ,C_{k-1}$ all lie in the same Fatou component.
\end{proposition}
\begin{proof}
We note that the orbits on the parabolic curves converge to $(t^p, t^q)$ along real directions $\alpha \cdot (t^q - t^p)$ satisfying
\begin{equation*}
G_k(\alpha \cdot (t^q, - t^p)) = \alpha^k (p-q)^k t^{kpq} (t^p, - t^q) = - \alpha \cdot (t^q, - t^p),
\end{equation*}
which gives
\begin{equation*}
\alpha^{k-1} = \frac{-1}{(p-q)^k t^{kpq}}.
\end{equation*}

Hence we see that if $p\cdot q$ and $k-1$ are relatively prime, and therefore $kpq$ and $k-1$ are also relatively prime, then if $t$ moves in a full circle around the origin then the $k-1$ parabolic curves have been permuted in a full cycle. As the parabolic curves vary smoothly this implies that the $k-1$ parabolic curves lie in the same Fatou component.

Suppose on the other hand that $pq$ is divisible by $k-1$.  Then following the parabolic curves as $t$ makes one full circle around the origin leaves the parabolic curves invariant. However, loops around the origin generate the fundamental group of $\mathrm{Reg}(V)$, hence the monodromy group is trivial. It follows that two parabolic curves $C_1$ and $C_2$ touching $V$ at the same point $z$ lie in the same Fatou component if and only if $C_1 = C_2$.
\end{proof}

 \end{document}